\documentclass[graybox]{svmult}

\usepackage{type1cm}        % activate if the above 3 fonts are
\usepackage{makeidx}         % allows index generation
\usepackage{graphicx}        % standard LaTeX graphics tool
\usepackage{multicol}        % used for the two-column index
\usepackage[bottom]{footmisc}% places footnotes at page bottom

\makeindex             % used for the subject index
\usepackage{amsmath,mathtools}
\usepackage{amsfonts}
\usepackage{amssymb}

\spdefaulttheorem{assumption}{Assumption}{\bfseries}{\itshape}
\newcommand{\N}{\mathbf{N}}
\newcommand{\R}{\mathbf{R}}
\newcommand{\C}{\mathbb{C}}
\newcommand{\Hc}{\mathcal H}
\begin{document}

\title*{On the controllability in projections for linear quantum systems}
\author{Nabile Boussa\"id, Marco Caponigro, and Thomas Chambrion}
\institute{Nabile Boussa\"id \at Universit\'e Bourgogne Franche-Comt\'e,
Laboratoire de math\'ematiques de Besan\c{c}on, 16 route de Gray, 25030
Besan\c{c}on Cedex, France. \email{nabile.boussaid@univ-fcomte.fr}
\and 
Marco Caponigro \at Dipartimento di Matematica, Universit\`a di Roma ``Tor Vergata'', Via della Ricerca Scientifica 1, 00133 Roma, Italy \email{caponigro@mat.uniroma2.it}
\and 
Thomas Chambrion \at 
Universit\'e Bourgogne Franche-Comt\'e,
Institut mathématique de Bourgogne, 9 avenue Alain Savary, BP47870, 21078 Dijon, France
\email{Thomas.Chambrion@u-bourgogne.fr}}

\maketitle

\abstract{
We present sufficient conditions for the exact controllability in projection of the linear Schr\"odinger equations in the case where the spectrum of the free Hamiltonian is pure point. We consider the general case in which the Hamiltonian may be nonlinear with respect to the control. The controllability result applies, in particular,
to Schr\"odinger equations with a polarizability term. 
}

\section{Introduction}

\subsection{Control of linear quantum systems}

In a complex Hilbert space $\Hc$ with hermitian product $\langle\cdot, \cdot\rangle$, we consider the linear Schr\"odinger equation
\begin{equation}\label{eq:main}\tag{LS}
\begin{cases}
\partial_t \psi =-i  H(u)\psi,\\
\psi(0)=\psi_0,
\end{cases}
\end{equation}
where 
$\psi(t)$ is the state of the system at time $t$, 
$\psi_0 \in  \Hc$ is the given initial datum,
$H(u)$ is a self-adjoint linear operator in $\Hc$ depending on a real-valued function $u$.
The operator $H(u)$ is usually called \emph{Hamiltonian}.
The parameter $u$ is a control
 representing our ability to influence the system and accounts for an external, appropriately shaped, electromagnetic field.
Equation~\eqref{eq:main} models many quantum systems under
the assumptions that the
interaction with the external environment (i.e. decoherence) is negligible and  that the control $u$ can be treated
as a classical field.

The aim of quantum control is to 
design a (open-loop) \emph{control} $u:[0,T]\to \R$ such that the solution of 
\eqref{eq:main}  satisfies desirable properties, for instance: the final point 
$\psi(T)$ is close, in some sense, to a given target.  

The controllability of system~\eqref{eq:main} is a well-established topic when the state space $\Hc$ is 
finite-dimensional, i.e. when the quantum mechanical system under consideration can
be adequately approximated as a system having a finite number of energy levels, see, for instance~\cite{dalessandro-book} or~\cite{borzibook} and references therein. 
Many controllability results in the finite-dimensional framework rely on general controllability methods for left-invariant control systems on compact Lie groups (\cite{JS72, JK81, GB82, AGK96}).

When the state space $\Hc$  is infinite-dimensional the situation is complicated by the subtleties of the evolution in Banach spaces and the consequent fragmental nature of controllability theory for PDEs.
For this reason, most of the current literature focuses on closed (i.e., conservative) systems and on the dipolar approximation, i.e. when the Hamiltonian $H(u) = H_0 + u H_1$ is linear with respect to the control.
For this kind of, so called, bilinear systems, one of the first known results is a negative one: indeed when $H_1$ is a bounded operator the bilinear Schr\"odinger equation is not exactly controllable, namely the attainable set has empty interior as a meagre set~\cite{BMS,turinici}.  
The obstruction to exact controllability holds even when considering very large class of controls~\cite{UP}, as for instance 
$L^1$ controls~\cite{BMS_CDC} or Radon measures~\cite{Guide_CDC}.
In certain cases, it is possible to prove exact controllability for the potential well in suitable 
functional spaces on a real interval (see~\cite{beauchard-coron, camillo, MN15}). 
The results extend to a system describing a particle confined on a radially symmetric $2$D domains~\cite{Moyano}. However in higher dimension and for more general systems the exact description of the reachable set seems a difficult task. The literature hence focuses on
weaker controllability properties.
Approximate controllability results have been obtained with different techniques:
adiabatic control (\cite{Boscain_Adami, adiabatiko}), Lyapunov methods (\cite{mirrahimi-continuous, Nersy, nersesyan, fratelli-nersesyan}), and Lie-algebraic methods (\cite{Schrod, rangan, Schrod2, periodic, BCS14, KZSH14, PS15, Esatta}).

Although the dipolar approximation usually gives excellent results for low intensity fields, it is sometimes necessary, 
when dealing with stronger fields, to consider a better approximation of $H(u)$ involving more terms of its expansion 
in $u$. 
In the control of orientation 
of a rotating HCN molecule, for instance, the
 model involves a quadratic term~\cite{DBAKU, DBAKU2} call polarizability term, and $H(u) = H_0 + u H_1 + u^2 H_2$.
This question has been tackled with Lyapunov methods for finite dimensional approximations in~\cite{Coron,Grigoriou} 
and in \cite{Morancey} for the infinite dimensional version for a class of regular systems. 

In this paper we study the general case when the Hamiltonian $H(u)$ is nonlinear in the control. We prove that, under generic condition, it is possible to control~\eqref{eq:main} approximately in a very strong sense. 
Theorem~\ref{thm:main} states that for every dimension $n$ of the Galerkin approximation and every initial and target condition, there exists a
control steering in finite time the infinite-dimensional control system from the initial condition
to a final state having the same first $n$ coordinates as the target (as long as the remaining coordinates do not vanish simultaneously). In quantum systems, this
problem is quite natural, since in physical experiments one can measure with accuracy only the
low-energy states. From the mathematical viewpoint, this kind of study is particularly interesting since the regularity of system, needed in order to define the solution, is, in fact, an obstacle to exact controllability (see, for instance~\cite{UP}). Exact controllability in
projection has been introduced and proved using geometric techniques based on controllability
result for the Galerkin approximations in~\cite{Navier-Stokes} for the $2$D Euler and Navier–Stokes equations
(see~\cite{shirikyan} for the $3$D case).

Our approach is based on Lie-algebraic techniques.
Lie-algebraic methods usually provide intrinsic, robust, and sometimes explicit control results for quantum mechanical systems. For these reasons, these techniques represent the main tool in the controllabilty of finite dimensional quantum systems.  In the infinite-dimensional case, however,
even the extension of the notion of Lie-bracket to infinite dimensional operators is not trivial.
While there are Lie-algebraic based control results in the presence of bounded operator~\cite{BB14} when dealing with unbounded operators the notion of Lie algebra is not well-defined in general.
The Lie--Galerkin condition developed in~\cite{BCS14} for bilinear systems combines a Lie-algebraic finite-dimensional condition with the Galerkin method to find sufficient
conditions for approximate controllability. Indeed the underlying idea of this
Lie–Galerkin technique is to choose a suitable sequence of Galerkin approximations, then to
use finite-dimensional geometric control tools in order to prove strong controllability (in some
suitable sense) for each Galerkin approximation, and finally to show how these controllability
properties provide approximate controllability for the original infinite-dimensional system.
This condition has been used to prove exact controllability in projections with piecewise constant controls for the bilinear Schr\"odinger equation~\cite{Esatta}. In this paper we extend the notion of Lie--Galerkin condition to infinite dimensional linear systems and prove 
 exact controllability in projections for~\eqref{eq:main} by means of piecewise constant controls taking only two values.
The main tool used in the proof is the continuity of the propagators of bilinear systems. This fact, combined with the approximation results of Section~\ref{sec:bangbang} allows to to infer controllability of the bilinear system with ``bang-bang'' controls, i.e. piecewise constant controls taking only two values allows, which in turns implies approximate controllability for~\eqref{eq:main}.

The proof of the main result, in Section~\ref{sec:linear} is based on a refined analysis on the controllability of a bilinear system presented in Section~\ref{sec:BSE}. The result is then applied to the controllability of systems with a polarizability term in Section~\ref{SEC_examples}.

\subsection{Framework and definition of propagators}\label{SEC_well_posedness}

The well-posedness of system \eqref{eq:main} when $H(u)$ is an unbounded operator on a infinite-dimensional space $\Hc$ is, in general, not trivial.
In order to define the solution we assume the following condition.

\begin{assumption}\label{ass:minimal}
The operators $H(0)$ and $H(1)$ are self-adjoint. $H(0)$  has 
pure point spectrum with an associated orthonormal basis $\Phi$ of eigenvectors 
and $H(1)-H(0)$ is bounded.
\end{assumption}

We denote by $\Phi=(\phi_k)_{k\in \N}$ 
 the complete orthonormal family of eigenvectors of $H(0)$ and by $(\lambda_k)_{k \in \N}$ the associated eigenvalues (that is, for all $k$ in $\N$, $H(0)\phi_k =\lambda_k \phi_k$).

Under Assumption~\ref{ass:minimal} it is possible to define the propagator  $\Upsilon^u$ of $-iH(u)$ and, hence, the solution of ~\eqref{eq:main} associated with a piecewise constant control $u:[0,T] \to \{0,1\}$.
Indeed,
 since $H(1)-H(0)$ is bounded
 the domain $D(H(1))$ of the self-adjoint operator $H(1)$ contains 
the domain $D(H(0))$ and, in particular,
every eigenvector $\phi\in\Phi$.
Therefore,
if 
$$
u=\sum_{j=1}^{m} u_{j}\mathbf{1}_{[t_{j-1},t_j)}, \quad u_j = 0,1,
$$
for $0=t_0<t_1 < \dots <t_m$,
then, one can define the associated propagator as
\begin{equation}\label{eq:defsolutionbangbang}
 \mathcal{Y}^u_{t}=e^{-i(t-t_k)H(u_k)} \circ e^{-i(t_k-t_{k-1})H(u_{k-1})} \circ
 \cdots \circ e^{-i t_1 H(u_1)},
\end{equation}
where 
$t_k\leq t<t_{k+1}$.

\subsection{Notation}

For $N \in \N$, $\mathbb{M}_{N}(\C)$ is the set of $N\times N$ matrices with entries in $\C$. 
The identity matrix of order $N$ is $I_N$.
The group of special unitary matrices $SU(N)$ is  
\[ SU(N)= \{M\in \mathbb{M}_N (\mathbf{C})| ~\overline{M}^T M =I_N \mbox{ and } \det M =1\},\]
which is a Lie group whose Lie algebra is 
 \[\mathfrak{su}(N)=\{M \in  \mathbb{M}_N (\mathbf{C}) | ~\overline{M}^T + M= 0 \mbox{ and } TrM=0\}. \]
This Lie algebra, seen as a real linear space, has dimension $N^2-1$.
We denote by $U(\Hc)$ the set of unitary operators on $\Hc$.

\subsection{Main result}
For every $n$ in $\N$, we define 
$$
{\cal L}_n^\Phi=\mathrm{span}\{\phi_1,\ldots,\phi_n\},
$$
and the projection of $\Hc$ on ${\cal L}_n^\Phi$, namely
\begin{equation}\label{eq:projectionphi}
 \begin{array}{llcl}
\Pi^{\Phi}_n&:\Hc & \rightarrow & \Hc\\
      & \psi & \mapsto & \sum_{j=1}^{n} \langle \phi_j,\psi\rangle \phi_j\,.
\end{array}
\end{equation}
The compression of a linear operator $Q$ on $\Hc$, with 
${\cal L}_n^\Phi \subset D(Q)$,
is denoted by
$$
Q^{(\Phi,n)}=\Pi^{\Phi}_n Q_{\upharpoonright {\mathcal L}_n^\Phi}.
$$

\begin{remark}
The projections in~\eqref{eq:projectionphi} and, as a consequence, the compressions of operators strongly depend on the basis $\Phi$. However, for the sake of readability
 we drop the mention to $\Phi$.
\end{remark}

\begin{remark}
When it does not create ambiguities we identify $\mathrm{Im}(\Pi_{n}) = {\cal L}_n^\Phi$ with $\C^{n}$. Given a linear operator $Q$ on $\Hc$ we identify the linear operator
$Q^{(n)} = \Pi_{n} Q \Pi_{n}$ preserving 
$\mathrm{span}\{\phi_{1},\ldots, \phi_{n}\}$ with 
its  
$n \times n$ complex matrix representation with 
respect to the basis $(\phi_{1},\ldots, \phi_{n})$.
\end{remark}

Define
$$
H_0^{(n)} = \Pi_{n} H(0) \Pi_{n} \quad \mbox{ and } \quad H_1^{(n)} = \Pi_{n} H(1) \Pi_{n}.
$$
Let us introduce the set $\Sigma_{n}$ 
of spectral gaps associated with the first $n$ eigenvalues of $H(0)$ as
$$
\Sigma_n = \{|\lambda_l - \lambda_k| \mid l,k = 1, \ldots, n\}.
$$
For every $\sigma \geq 0$, every $m\in\N$, and every $m\times m$ matrix $M$, let
$$
{\cal E}_\sigma(M)=  (M_{l,k} \delta_{\sigma, |\lambda_l - \lambda_k|})_{l,k=1}^m,
$$
where $\delta_{\cdot,\cdot}$ denotes the Kronecker symbol. The $n\times n$ matrix ${\cal E}_\sigma(H_1^{(n)})$, corresponds then to the  ``selection'' in $H_1^{(n)}$ of the spectral gap $\sigma$:  
the $(l,k)$-elements such that $|\lambda_{l}-\lambda_{k}|\neq \sigma$ are set to $0$.

Define
\begin{align}%\label{sigmabar}
\Xi_n=\big\{ \sigma \in \Sigma_n & \mid
(H_1)_{k,l}\delta_{\sigma, |\lambda_l - \lambda_k|} =0, 
\nonumber\\
&\qquad\mbox{ for every }
k=1,\ldots,n \mbox{ and }  l>n\big\}.\label{sigmabar}
\end{align}

The set $\Xi_n$ can be seen as follows: If $\sigma\in\Xi_n$ then the matrix
$M =\mathcal{E}_\sigma(H_1^{(n)}) $ is such that 
$$
 \mathcal{E}_\sigma(H_1^{(N)}) = 
\left(
\begin{array}{c|c}
M&0\\ \hline 0 & *
\end{array}
\right)
\mbox{ for every } N>n,
$$
or, which is equivalent,
\begin{equation}\label{eq:comm}
\left[
\Pi_n, \mathcal{E}_\sigma(H_1^{(N)})
\right]=0\qquad \mbox{ for every } N>n.
\end{equation}

In particular $\mathrm{span}\{\phi_{1},\ldots,\phi_{n}\}$ is invariant for the evolution of ${\cal E}_\sigma(H_1^{(N)})$ for every $N>n$.
The spectral gaps $\sigma \in \Xi_n$ are, therefore, those for which the selections ${\cal E}_\sigma(H_1^{(n)})$ 
% are decoupled 
define finite dimensional dynamics of order $n$ decoupled from the infinite dimensional evolution.

The main assumption of this paper, introduced in the definition below, is the extension of the Lie--Galerkin condition~\cite{BCS14} to system~\eqref{eq:main}.

\begin{definition}\label{def:CCC}
For every $n \in \N$ define 
\begin{align*}
 \mathcal{M}_{n} = 
\left\{iH_0^{(n)} 
\right\}
\cup \left\{i{\cal E}_\sigma(H_1^{(n)})\mid
\sigma \in \Xi_n \right\}.
\end{align*}
We say that the \emph{Lie--Galerkin condition}\ holds for system~\eqref{eq:main} if 
for every $n_0\in \mathbb{N}$ there exists $n> n_0$ 
such that $0\in \Xi_n$ and
\begin{equation}\label{hypothesis}
\mathrm{Lie}\mathcal{M}_{n}  \supseteq \mathfrak{su}(n).
\end{equation}
\end{definition}

We can now state the main result of exact controllability in projections.

\begin{theorem}\label{thm:main}  
Let system~\eqref{eq:main} satisfy Assumption~\ref{ass:minimal} and
the Lie--Galerkin  condition.
Then, for every $N\in \N$, $\psi_0, \psi_1\in \mathrm{span}(\Phi)$, with $\|\psi_0\| = \|\psi_1\|=1$ and
$
\|\Pi_N(\psi_1)\| < 1
$ 
there exists a piecewise constant control $u:[0,T] \to \{0,1\}$ such that
$$
\Pi_N(\mathcal{Y}^u_T(\psi_0)) = \Pi_N(\psi_1). 
$$
\end{theorem}

\section{The Bilinear Schr\"odinger Equation}\label{sec:BSE}

In this section we focus on the case $H(u) = iA  + u iB$, that is when $H$ is linear with respect to the control $u$. The corresponding system is usually called bilinear because it is linear with respect to the state and with respect to the control. 
We aim at proving that the bilinear case is exactly controllable in projection by means of bang-bang controls, 
i.e. piecewise constant controls taking value $0$ and $1$. We will prove in Section~\ref{sec:linear} that this 
result implies Theorem~\ref{thm:main}.

Consider the system
\begin{equation}\label{eq:BSE}\tag{BSE}
\begin{cases}
\partial_t \psi = A\psi + u(t) B\psi,\\
\psi(0)=\phi_0 \in  \Hc,
\end{cases}
\end{equation}
satisfying the following assumption.
\begin{assumption}\label{ass:bls}
The pair of operators $(A,B)$ 
is such that
\begin{itemize}
\item The skew-adjoint operator $A$ has pure point spectrum with an associated complete orthonormal basis $\Psi$ of eigenvectors ;
\item The operator
$B$ is skew-symmetric and bounded.
\end{itemize}
\end{assumption}

 By
Assumption~\ref{ass:bls}
for every  $u\in \R$, $A+uB$ defined on ${\rm Span}(\Psi) \subset D(A)$ is essentially skew-adjoint.
We can, therefore, define the solution of~\eqref{eq:BSE} associated with a piecewise constant control 
$$
u=\sum_{j=1}^{m} u_{j}\mathbf{1}_{[t_{j-1},t_{j})}
$$
with $u_j\in [0,\delta]$, for $j=1,\ldots, m$,  $0=t_0<t_1 < \dots <t_m$
as the concatenation 
$$
 \Upsilon^u_{t}=e^{(t-t_k)(A+u_k B)} \circ e^{(t_k-t_{k-1})(A+u_{k-1}B)} \circ
 \cdots \circ e^{t_1(A+u_1B)},
$$
where $t_k\leq t < t_{k+1}$.
In analogy with the definition of solution associated with a bang-bang control~\eqref{eq:defsolutionbangbang}.

\begin{remark}\label{rk:continuitybilinear}
One of the main advantages of dealing with bilinear control systems is that, when the control operator $B$ is bounded on $\Hc$, 
then the propagator $\Upsilon$ of the system is continuous with respect to the control (see~\cite{BMS,UP}) in the sense:
If $(u_k)_{k\in \N}$ is a sequence of piecewise constant controls, $u_n:[0,T] \to [0,\delta]$, converging in $L^1([0,T])$ to a piecewise constant control 
$u:[0,T] \to [0,\delta]$ then
$$
\Upsilon_t^{u_k}\psi \to \Upsilon_t^u\psi \quad \mbox{ as } k \to \infty,
$$
uniformly in $t\in [0,T]$ for every $\psi \in \Psi$.
\end{remark}

We say that system~\eqref{eq:BSE} satisfies the Lie--Galerkin  condition if the system~\eqref{eq:main} with $H(0)=i A$ and $H(1)$ as the closure of $i(A+B)$ satisfies the Lie--Galerkin  condition.

\begin{theorem}\label{thm:bilinear}
Let $(A,B)$ satisfy Assumption~\ref{ass:bls} and let system~\eqref{eq:BSE} satisfy the Lie--Galerkin  condition.
Then, for every $N\in \N$, $\psi_0, \psi_1\in \mathrm{span}(\Phi)$, with $\|\psi_0\| = \|\psi_1\|=1$ and
$
\|\Pi_N(\psi_1)\| < 1
$ 
there exists a piecewise constant control $u:[0,T] \to \{0,1\}$ such that
$$
\Pi_N(\Upsilon^u_T(\psi_0)) = \Pi_N(\psi_1). 
$$
\end{theorem}

\subsection{Bang-bang approximation of piecewise constant functions}\label{sec:bangbang}

In this section we prove that it is possible to approximate 
the propagator of a particular finite dimensional system associated with piecewise constant controls with the one associated with suitable controls taking only two values. As mentioned, this approximantion, together with the continuity of the propagators is crucial in the  extension of the controllability results from bilinear systems to linear ones.
%~\eqref{eq:BSE} to~\eqref{eq:main}.    %Borrowing a leaf from optimal control theory we call the one associated with 

\begin{lemma}\label{LEM_tripoint}
Let $u:[0,T]\to [0,\delta]$, $\delta>0$ be a piecewise constant function and let $a>\delta$. There exists a sequence $(w_k)_{k \in \N}$   of piecewise constant functions $w_k:[0,T]\to \{0,a\}$ uniformly bounded in $L^1([0,T],\R)$ by $\|u\|_{L^1}$, and continuously depending on $u$ for the $L^1$ topology, such that
$$
\int_0^t w_k(s) ds \to \int_0^t u(s) ds, \quad \mbox{ as } k \to \infty,
$$
uniformly for $t$ in $[0,T]$. 
\end{lemma}

\begin{proof}
We prove the convergence on $[0,T)$. Lemma~\ref{LEM_tripoint} follows on $[0,T]$ by continuity.
Let $k$ in $\N$ and let us divide the interval $[0,T)$ in $k$ intervals $I_h:=[hT/k, (h+1)T/k)$ for $h=0,\dots, k-1$.
Let 
$$
U_h = \int_{I_h} u(s) ds.
$$
We define the piecewise constant function $w_k:[0,T]\to \{0,a\}$  as follows
$$
w_k(t) =
\left\{
\begin{array}{ll}
0 & \mbox{for } t \in \left[\frac{hT}{k},\frac{(h+1)T}{k} - \frac{U_h}{a} \right),\\
a & \mbox{for } t \in \left[\frac{(h+1)T}{k} - \frac{U_h}{a}, \frac{(h+1)T}{k} \right),\\
0 &  \mbox{for } t \notin I_h.
\end{array}
\right.
$$
Hence, for every $h=0,\dots, k-1$, we have
$$
\int_{I_h} w_k(s)ds = U_h =  \int_{I_h} u(s) ds.
$$
Now for every $t \in [0,T]$ let $h_t$ be the smallest integer such that 
$(h_t+1) T \geq t k$, then
\begin{align*}
\left|\int_{0}^t w_k(s)ds - \int_{0}^t u(s) ds\right| 
& = 
\left|\int_{\frac{h_t T}{k}}^{t} (w_k(s)- u(s))ds\right| \\
& \leq (\delta + a)\frac{T}{k}.
\end{align*}

Finally, notice that
\begin{align*}
\int_0^T |w_k(s)| \mathrm{d}s & = \sum_{h=0}^{k-1}
 \int_{I_h} |w_k(s)| \mathrm{d}s  \\
 & =  
\sum_{h=0}^{k-1}
 \left | \int_{I_h} u(s) \mathrm{d}s \right | \\
 &  \leq  \sum_{h=0}^{k-1} \int_{I_h} |u(s)| \mathrm{d}s = \int_0^T |u(s)| \mathrm{d}s,
\end{align*}
which ensures the boundedness in $L^1$ and concludes the proof.
\end{proof}

Let $\Theta(t) = e^{-tA}Be^{tA}$. For every $N \in \N$ consider 
$$
\Theta^{(N)}(t) = \Pi_N \Theta(t)\Pi_N = e^{-tA^{(N)}}B^{(N)}e^{tA^{(N)}}.
$$
Then, for every piecewise constant function $v$ let us denote by $X^{(N)}_{t,s}(v)$, the propagator of 
$v(\cdot)\Theta^{(N)}(\cdot)$.

\begin{lemma}
Let $u:[0,T]\to [0,\delta]$, $\delta>0$ be a piecewise constant function and let $a>\delta$. There exists a sequence $(w_k)_{k \in \N}$of piecewise constant functions $w_k:[0,T]\to \{0,a\}$ uniformly bounded in $L^1([0,T],\R)$ by $\|u\|_{L^1}$, such that,
for every $N \in \N$,
$$
X^{(N)}_{t,s}(w_k) \to X^{(N)}_{t,s}(u),  \quad \mbox{ as } k \to \infty,
$$
 uniformly for $s, t$ in  $[0,T]$.
\end{lemma}

\begin{proof}
Let $(w_k)_{K \in N}$ be the sequence of controls whose existence is given by Lemma~\ref{LEM_tripoint}. For every $N \in \N$ and for every $t> 0$,
by integration by parts, 
\begin{align*}
  & \left \|\int_0^t u(s) \Theta^{(N)}(s) \mathrm{d}s -
\int_0^t w_k(s) \Theta^{(N)}(s) \mathrm{d}s \right\| \\
  & \quad\leq  \left \| \left (\int_{\tau=0}^t (u(\tau)-w_k(\tau)) \mathrm{d}\tau \right ) \Theta^{(N)}(t)\right\| \\
   &\qquad + \left \|\int_{s=0}^t   \left ( \int_{\tau=0}^s (u(\tau)-w_k(\tau)) \mathrm{d}\tau \right ) 
 e^{-sA^{(N)}} [B^{(N)},A^{(N)} ]   e^{sA^{(N)}}  \mathrm{d}s\right\|.
\end{align*}

Hence, by Lebesgue Dominated Convergence Theorem, since $(\int_{0}^{t} w_k(s)\mathrm{d}s)_{k \in \N}$ converges   to  $\int_{0}^{t} u(s)\mathrm{d}s$ uniformly for $t$ in  $[0,T]$, then  
$$
\int_0^t w^k(s) \Theta^{(N)}(s) \mathrm{d}s  \to \int_0^t u(s) \Theta^{(N)}(s) \mathrm{d}s, \quad \mbox{ as } k \to \infty,
$$
which, in turn, implies (see, for instance, \cite[Lemma~8.2]{AS04}) that
%$$
%\rexp{t}{w_k(s)\Theta^{(N)}(s)}{s} \to \rexp{t}{u(s)\Theta^{(N)}(s)}{s}  \quad \mbox{ as } k \to \infty,
%$$
$$
X^{(N)}_{t,s}(w_k) \to X^{(N)}_{t,s}(u),  \quad \mbox{ as } k \to \infty,
$$
 uniformly for $s, t$ in  $[0,T]$.
\end{proof}

 \begin{lemma}\label{lem:ipsilon}
 Let $(A,B)$ satisfy Assumption~\ref{ass:bls}.
Let $u:\R \to \R$ be a piecewise constant function and $\phi \in \Hc$. 
Then
$$
\frac{d}{dt} \langle \psi, e^{-tA} \Upsilon^u_t(\phi) \rangle = 
- \langle u(t)\Theta(t)\psi, e^{-tA} \Upsilon^u_t(\phi) \rangle,
$$
for every $\psi \in \Psi$ and for almost every $t\in \R$.
\end{lemma}

\begin{proof}
Let
$$
y(t) = e^{-t A} \Upsilon^u_{t}(\psi). 
$$
Let $i\lambda$ be the eigenvalue of $A$ associated with $\psi \in \Psi$, namely $A\psi = i \lambda \psi$. 
Then
\begin{align*}
 \frac{d}{dt} \langle \psi, y(t) \rangle & = 
  \frac{d}{dt}   e^{i t \lambda} \langle \psi, \Upsilon^u_{t}(\phi)\rangle\\
  & =  i \lambda  e^{i t \lambda} \langle \psi, \Upsilon^u_{t}(\phi)\rangle
   +  e^{i t \lambda} \frac{d}{dt}   \langle \psi, \Upsilon^u_{t}(\phi)\rangle
    \\
    & = i \lambda  e^{it \lambda} \langle \psi,\Upsilon^u_{t}(\phi)\rangle
   -  e^{i t \lambda} 
    \langle (A + u(t) B)\psi,  \Upsilon^u_{t}(\phi)\rangle
    \\
&=  
- e^{it \lambda}\langle 
   u(t) B
\psi,  e^{tA}y(t)\rangle\\
& = - \langle u(t)\Theta(t)\psi,  y(t)\rangle,
\end{align*}
 for almost every $t\in \R$.
\end{proof}

 \subsection{Tracking of admissible matrices}

Recall that system~\eqref{eq:BSE} satisfies the Lie--Galerkin condition. 
For $n_0 \in \N$ let  $n > n_0$ be given by the Lie--Galerkin condition.
Consider the collection of matrices
\begin{align*}
 \mathcal{W}_{n}
 = &
\left\{A^{(n)} 
\right\}
\cup \left\{{\cal E}_0(B^{(n)})\right\}\\
&\cup \left\{{\cal E}_0(B^{(n)})+\nu{\cal E}_\sigma(B^{(n)})\mid
\sigma\in \Xi_n,\sigma\ne 0, \nu \in (-1/2,1/2)
\right\},
\end{align*}
where 
$\Xi_n$  is defined in
\eqref{sigmabar}.
Notice that for every $\sigma \in \Xi_n$~\eqref{eq:comm} holds true for the operator $\mathcal{E}_\sigma(B^{(N)})$, namely, 
$\left[\Pi_n, \mathcal{E}_\sigma(B^{(N)})
\right]=0$ for every $N>n$ and for every $\sigma$ in $\Xi_n$.

\begin{proposition}\label{prop:convergencepropagators}
Let $(A,B)$ satisfy Assumption~\ref{ass:bls}.
Let $q \in \N$, $a,b \in \R$ with $0<a<b$, and $M_1, \ldots, M_q \in \mathcal{W}_n$.
For every $\varepsilon>0$ and
$\tau_1, \ldots, \tau_q \in [a,b]$
there exist $w:[0,T] \to \{0,1\}$ piecewise constant and $\gamma \geq 0$ such that 
$$
\|
\Pi_n\Upsilon^{w}_{T} -   e^{\tau_q M_q} \circ \dots \circ e^{\tau_1 M_1}
\|_{L(\Pi_n(\Hc),\Hc)} < \varepsilon.
$$
Moreover, 
$w$ can be taken  
continuously depending on $\tau_1, \ldots, \tau_q \in [a,b]$.
\end{proposition}

\begin{proof}
Following~\cite[Proposition~4.1]{Esatta} we have, for every $\eta,\delta>0$, that there exist
$u:[0,T_u] \to [0,\delta]$ piecewise constant and $\gamma \geq 0$ such that 
\begin{equation}\label{eq:3141}
\|
\Upsilon^{u}_{T_u} - e^{\gamma A} \circ e^{\tau_q M_q} \circ \dots \circ e^{\tau_1 M_1}
\|_{L(\Pi_n(\Hc),\Hc)} < \eta,
\end{equation}
where $u$ depends continuously on $\tau_1, \ldots, \tau_q \in [a,b]$ and $\gamma$ is independent on $\tau_1, \ldots, \tau_q \in [a,b]$. Moreover the $L^1$-norm of $u$ is independent on $\eta$ (see \cite[Section 3.1]{Esatta} and~\cite{periodic}). Let $\delta<1$, then there exists $K>0$ such that
every $u$ in~\eqref{eq:3141} satisfies $\|u\|_{L^1} \leq K$.

Let $N>n$ be such that 
$$
\|\Pi_n B (I-\Pi_N)\| \leq \frac{\varepsilon}{10K}.
$$
The existence of such a $N$ is guaranteed by the compactness of $\Pi_n B$ (bounded with finite rank) using the fact that $\Psi\subset D(B)$ and that $B$ is skew--symmetric. 
Then, call $\xi = \|\Pi_N B (I-\Pi_N)\|$, %{\color{red} or $\|\Pi_n B (I-\Pi_N)\|$}, 
and consider the control $u$ associated with $\eta=\varepsilon/(10\max\{1,K\xi\})$ in~\eqref{eq:3141}.

Let $(w_k)_{k \in \N}$ be a piecewise constant functions $w_k:[0,T_u]\to \{0,1\}$ associated with (and continuously depending on) $u$, the existence of which is given by Lemma~\ref{LEM_tripoint}.
For every $\psi \in \Pi_n(\Hc)$ with norm 1, let
$$
y(t) = e^{-tA}\Upsilon_t^{u}(\psi) \quad \mbox{ and } \quad 
y_k(t) = e^{-tA}\Upsilon_t^{w_k}(\psi).
$$

We deduce from Lemma~\ref{lem:ipsilon} and from variation of constants formula that  for any $\psi\in{\cal L}_N$
\begin{align*}
 \Pi_n e^{-tA}\Upsilon^v_t \psi &= 
\Pi_n
X^{(N)}_{t,0}(v)\left(\psi\right)\\
&+ 
\Pi_n\int_0^t
X^{(N)}_{t,s}(v)
\Pi_N 
v(s)
\Theta\left(s\right)
(\mathrm{I}-\Pi_N) 
e^{-sA}\Upsilon^v_s \psi
ds.
\end{align*}
Hence
\begin{align*}
\|\Pi_n&(y_k(t) - y(t))\| \leq
\|\Pi_n\left(X^{(N)}_{t,0}(u) - X^{(N)}_{t,0}(w_k)\right)\psi\|
\\
&\quad + 
	\|\Pi_n \int_{s=0}^t
	\left(
	X^{(N)}_{t,s}(w_k) - X^{(N)}_{t,s}(u)
	\right)
	\Pi_N w_k(s)\Theta(s) (I-\Pi_N) y_k(s)
	ds\|
\\
&\qquad + 
\|
\Pi_n
\int_{s=0}^t X^{(N)}_{t,s}(u) \Pi_N\Theta(s) (I-\Pi_N) (u(s)y(s) - w_k(s)y_k(s))ds
\|.
\end{align*}
Now 
\begin{align*}
\int_{s=0}^t
&
\Pi_n
 X^{(N)}_{t,s}(u) \Pi_N\Theta(s) (I-\Pi_N) (u(s)y(s) - w_k(s)y_k(s))ds
=\\
&=\quad \int_{s=0}^t
 X^{(N)}_{t,s}(u) \Pi_n\Theta(s) (I-\Pi_N) (u(s)y(s) - w_k(s)y_k(s))ds
 \\
&\qquad -
\int_{s=0}^t
 \left[X^{(N)}_{t,s}(u),\Pi_n\right] \Pi_N\Theta(s) (I-\Pi_N) (u(s)y(s) - w_k(s)y_k(s))ds.
\end{align*}

Notice that by~\eqref{eq:3141} and for the definition of the set 
$\mathcal{W}_n$ and of $\Xi_n$ in~\eqref{sigmabar} we have that
$$
\sup_{s<t}\left\|\left[X^{(N)}_{t,s}(u),\Pi_n\right] \right\|< \eta=\varepsilon/(10\max\{1,K\xi\}),
$$
indeed $X^{(N)}_{t,s}(u)$ is ``close'' (see also~\cite[(3.8)]{Esatta}) to the composition of exponential of matrices in $\mathcal{W}_n$ satisfying~\eqref{eq:comm}. 

Then, 
since $\Pi_N \Theta(t) (I-\Pi_N)$ is an operator uniformly bounded with respect to $t \in \R$, 
\begin{align*}
\|
\Pi_n
&
\int_{s=0}^t  X^{(N)}_{t,s}(u) \Pi_N\Theta(s) (I-\Pi_N) (u(s)y(s) - w_k(s)y_k(s))ds
\|\\
&\leq 
\|\Pi_n B (I-\Pi_N)\| (\|u\|_{L^1}+\|w_k\|_{L^1})\\
&\qquad
+ \sup_{s<t}\left\|\left[X^{(N)}_{t,s}(u),\Pi_n\right]\right\| \|\Pi_N B (I-\Pi_N)\|(\|u\|_{L^1}+\|w_k\|_{L^1})\\
& 
\leq \frac{\varepsilon}{5} + \frac{\varepsilon}{5}.
\end{align*}
Finally consider $k \in \N$ sufficiently large such that
$$
\|\Pi_n\left(X^{(N)}_{t,0}(u) - X^{(N)}_{t,0}(w_k)\right)\psi\| < \frac{\varepsilon}{5\max\{K\xi,1\}}.
$$
Hence
$$
\|\Pi_n \int_{s=0}^t
	\left(
	X^{(N)}_{t,s}(w_k) - X^{(N)}_{t,s}(u)
	\right)
	\Pi_N w_k(s)\Theta(s) (I-\Pi_N) y_k(s)
	ds\| \leq  \frac{\varepsilon}{5}.
$$
In conclusion 
$$
\|\Pi_n(y_k(t) - y(t))\| <  \frac{4}{5}\varepsilon.
$$
uniformly with respect to $t \in [0,T_u]$. Since $e^{tA}$ is unitary and $\psi$ arbitrary we have that
$$
\|\Pi_n \left(\Upsilon^{w_k}_t - \Upsilon^{u}_t\right)\Pi_n\|<  \frac{4}{5}\varepsilon.
$$
Hence
\begin{align*}
\|
\Pi_n&\Upsilon^{w}_{T_u} - e^{\gamma A} \circ e^{\tau_k M_k} \circ \dots \circ e^{\tau_1 M_1}
\|_{L(\Pi_n(\Hc),\Hc)}\leq  \\
&\leq  
\|
\Pi_n(\Upsilon^{w_k}_{T_u} 
- \Upsilon^{u}_{T_u})
\|_{L(\Pi_n(\Hc),\Hc)}
+ \\
& \qquad +
\|
\Pi_n\Upsilon^{u}_{T_u}
- e^{\gamma A} \circ e^{\tau_q M_q} \circ \dots \circ e^{\tau_1 M_1}
\|_{L(\Pi_n(\Hc),\Hc)}\\
&\quad \leq \frac{4}{5}\varepsilon + \frac{1}{5}\varepsilon = \varepsilon.
\end{align*}
By Poincar\'e Recurrence Theorem (see~\cite[Lemma~4.2]{Esatta}) it is then sufficient to consider as control $w$ the concatenation of the control $w_k$ on $[0,T_u]$ with a function constantly equal to $0$ for a certain time $T-T_u$.
\end{proof}

\subsection{A degree argument}
\label{SEC_SUB_SUB_Degree}

\begin{lemma}\label{LEM_stabilite_structurelle_surjection}
 Let $X \subset  \R^n$ be open and bounded and let $F \in  C(X, \R^n )$  be a homeomorphism between
$X$ and $F(X)$. 
Let $G \in C(X, \R^n )$
and $\varepsilon:=\max_{x \in \partial X} |F (x) - G(x)|$.
If $y_0$ is in $F(X)$ and such that $\varepsilon < \mathrm{dist}(y_0,F(\partial X))$  
then $y_0 \in \mathrm{int}(G(X))$.
\end{lemma}
Last result is standard in degree theory. The proof for the statement in this form can be found in~\cite[Lemma~5.2]{Esatta} or~\cite[Lemma~7]{AC_drift}.

\subsection{Normal controllability}\label{sec:normal}

Let 
$\psi_0, \psi_1 \in \mathrm{span}(\Phi)$ and 
with $
\|\Pi_N(\psi_1)\| < 1$ and consider $U\in U(\Hc)$ be such that $U\psi_0 = \psi_1$.
Let $n > N$ be such that the Lie--Galerkin  condition holds.
Let now $M \in SU(n)$ be such that
$M^{(N)}:=\Pi_{N}M\Pi_{N} = \Pi_{N}U\Pi_{N}=:U^{(N)}$.
Since, by the Lie--Galerkin condition
$$
\mathrm{Lie}\mathcal{W}_{n} \supset \mathfrak{su}(N),
$$
 we have classical results of \emph{normal controllability} (see~\cite{JS72} and~\cite[Theorem~4.3]{Sussmann76})
implying the existence of
$M_{1},\ldots,M_{k} \in\mathcal{W}_{n}$ and $t_{1},\ldots,t_{k} > 0$ such that the map
\begin{equation}\label{eq:9378}
E: (s_{1},\ldots,s_{k}) \mapsto e^{s_{k}M_{k}} \circ \dots \circ e^{s_{1}M_{1}},
\end{equation}
has rank $	n^2-1$ at $(t_{1},\ldots,t_{k} )$ and 
$$
E(t_{1},\ldots,t_{k} ) = M.
$$
Let us call, for simplicity, ${\nu} = n^2-1$.
By linear extraction there exist $j_{1},\ldots, j_{\nu} \in \{1,\ldots, k\}$ such that the map $F$ defined by
$$
(s_{j_{1}},\ldots,s_{j_{\nu}}) \mapsto \left( E(t_{1},\ldots,t_{j_{1}-1}, s_{j_{1}},t_{j_{1}+1}, \ldots,t_{j_{\nu}-1}, s_{j_{\nu}},t_{j_{\nu}+1},\ldots,t_{k}) \right)\Pi_{N}\psi_0,
$$
has rank $\nu$ at 
$(t_{j_{1}},\ldots,t_{j_{\nu}})$ and 
\begin{equation}
F(t_{j_{1}},\ldots,t_{j_{\nu}}) = M^{(N)}\Pi_N\psi_0 \quad \left(= \Pi_N\psi_1 \right).
\end{equation}
Now let $\varepsilon >0$ be such that
$$
X := B_{\varepsilon}(t_{j_{1}},\ldots,t_{j_{\nu}}) \subset (0,+\infty)^{\nu},
$$
where $B_{\varepsilon}(\bf t)$ (resp. $\overline{B_{\varepsilon}(\bf t)}$) is the open (resp. closed) ball of radius $\epsilon$ centered at ${\bf t} \in \R^{\nu}$. Then
$F$ is a diffeomorphism between $X$ and $F(X)$.
Let 
\begin{equation}\label{eq:eta}
 \eta = \inf_{(s_{1},\ldots,s_{\nu})\in \partial X} \|F(s_{1},\ldots,s_{\nu}) - \Pi_N\psi_1\|,
\end{equation}
and note that $\eta >0$.

\begin{lemma}\label{lem:PConto}
There exists a map associating with every 
$(s_1,\ldots, s_{\nu}) \in \bar X$ a piecewise constant control 
$w:[0,\infty) \to \{0,1\}$ and $T>0$ such that
if $G = \Pi_N \Upsilon^w_{T} \Pi_N\psi_0$ then
$$
\Pi_N\psi_1\in  \mathrm{int}(G(X)).
$$
%$$
% \begin{array}{cccc}
%G:& \bar X  &\to & \mathrm{span}\{\phi_{1}, \ldots, \phi_{N}\} \\
%&(s_1,\ldots,s_{\nu}) & \mapsto 
%&\Pi_{N}\Upsilon^{u}_{T} \Pi_{N}
%\end{array}
%$$
%contains $\Pi_{N}(T\psi_{1}), \ldots,\Pi_{N}(T\psi_{p}) $ in its interior.
\end{lemma}

\begin{proof}
Let $\eta$ be as in~\eqref{eq:eta}. By Proposition~\ref{prop:convergencepropagators} applied with $\varepsilon <  \eta$
there exists $w\in PC([0,\infty),\{0,1\})$ such that
$$
\sup_{(s_1,\dots, s_{\nu})\in\partial X}|F(s_1,\dots, s_{\nu})-G(s_1,\dots, s_{\nu})| < \eta.
$$
The continuity of the control $w$ on $(s_1,\dots, s_{\nu})$
is given by Proposition~\ref{prop:convergencepropagators} while the continuity of the propagator $\Upsilon^w_T$ on the control $w$ is given by~\cite[Corollary~9]{UP} (see also Remark~\ref{rk:continuitybilinear}).
The conclusion then follows from Lemma~\ref{LEM_stabilite_structurelle_surjection}.
\end{proof}

\section{Proof of the main result}\label{sec:linear}
We are now ready to prove the main result that is a consequence of Theorem~\ref{thm:bilinear}.

\begin{proof}[Proof of Theorem~\ref{thm:bilinear}]
Notice that, if
$$
\mathcal{V}_n = \left\{A^{(n)}\right\}
\cup \left\{{\cal E}_\sigma(B^{(n)})\mid \sigma \in \Xi_n\right\},
$$ 
then
$$
\mathrm{Lie}(\mathcal{V}_{n}) = \mathrm{Lie}(\mathcal{M}_{n}).
$$
Indeed for every $\sigma \in \Sigma_n$,
$$
[A^{(n)},\mathcal{E}_\sigma(B^{(n)})] = [-iH_0^{(n)},
\mathcal{E}_\sigma(-iH_1^{(n)}+i H_0^{(n)})]
=[iH_0^{(n)},
\mathcal{E}_\sigma(iH_1^{(n)})].
$$
Moreover if $0 \in \Xi_n$ then
$\mathrm{Lie}(\mathcal{W}_{n})=\mathrm{Lie}(\mathcal{V}_{n})$. In particular by the Lie--Galerkin condition, $\mathrm{Lie}(\mathcal{W}_{n}) \supseteq \mathfrak{su}(n)$ and the result follows from Lemma~\ref{lem:PConto}.
\end{proof}

\begin{proof}[Proof of Theorem~\ref{thm:main}]
Consider $A=-iH(0)$ and $B=-i(H(1)-H(0))$. Then the pair of operators $(A,B)$ satisfies 
Assumption~\ref{ass:bls}. Moreover if~\eqref{eq:main} satifsfies the Lie--Galerkin condition then 
so does \eqref{eq:BSE} by definition.
Theorem~\ref{thm:main} then follows directly from Theorem~\ref{thm:bilinear}
since 
$$
\Upsilon_t^u = \mathcal{Y}^u_t,
$$
for every control $u:\R \to \{0,1\}$ and for every $t \geq 0$. 
\end{proof}

\section{Control of the Schr\"odinger Equation with a polarizability term}\label{SEC_examples}

\subsection{Non-resonance condition}
A simple assumption implying the Lie-Galerkin condition is the existence of a non-resonant chain of connectedness. This notion, presented in~\cite{Schrod2} as a sufficient condition for approximate controllability
for bilinear systems can be extended to~\eqref{eq:main} under Assumption~\ref{ass:minimal} as follows.

\begin{definition}\label{def:nonresonant} 
We say that $S \subset \N^2$ \emph{couples} two levels $l,k$ in $\N$,
if
there exists a finite sequence $\big ((s^{1}_{1},s^{1}_{2}),\ldots,(s^{q}_{1},s^{q}_{2}) \big )$
in $S$ such that
\begin{description}
\item[$(i)$] $s^{1}_{1}=l$ and $s^{q}_{2}=k$;
\item[$(ii)$] $s^{j}_{2}=s^{j+1}_{1}$ for every $1 \leq j \leq q-1$;
\item[$(iii)$] $\langle  \phi_{s^{j}_{1}}, H(1) \phi_{s^{j}_{2}}\rangle \neq 0$ for $1\leq j \leq q$.
\end{description}

$S$ is called a \emph{connectedness chain} if $S$  couples every pair of levels in $\N$.

A connectedness chain is said to be \emph{non-resonant} if for
every $(s_1,s_2)$ in $S$, 
$
|\lambda_{s_1}-\lambda_{s_2}|\neq |\lambda_{t_1}-\lambda_{t_2}| 
$ for every $(t_{1},t_{2})$ in
$\N^2\setminus\{(s_1,s_2),(s_2,s_1)\}$ such that $\langle \phi_{t_{2}}, B \phi_{t_{1}}\rangle  \neq 0$.
\end{definition}

A system admitting a non-resonant chain of connectedness satisfies the Lie-Galerkin condition as the following Lemma states. 

\begin{lemma}\label{prop:nonresonante}
Assume that $\langle \phi_{l}, H(1)\phi_{k} \rangle= 0$ whenever
$l\neq k$ and $\lambda_{l} = \lambda_{k}$. If 
there exists a non-resonant connectedness chain coupling every pair of levels of $H(0)$ then system~\eqref{eq:main} satisfies the Lie--Galerkin condition.
\end{lemma}

\begin{proof}
The first assumption in the statement of the proposition
implies that 
$0 \in \Xi_n$ for every $n$. 
The non-resonance condition on the eigenvalues implies that, for every $n$, $\Xi_n = \Sigma_n$. 
Since, moreover, $\mathcal{E}_{|\lambda_l - \lambda_k|}(H_1^{(n)})$ has at most two nonzero entries, 
the Lie--Galerkin condition
follows from the existence of a connectedness chain (see, for instance the proof of~\cite[Proposition 3.1]{Schrod2} for details).
\end{proof}

\subsection{Generic bounded coupling  potentials}\label{SEC_bounded_potentials}

Let $\Omega$ be a compact Riemannian manifold or a bounded domain in $\mathbf{R}^n$. Let $V,W_1,W_2:\Omega \to \mathbf{R}$ be three measurable bounded 
functions. We consider the system
\begin{align}
\lefteqn{\mathrm{i}\frac{\partial \psi}{\partial t}(x,t)=
(-\Delta +V(x))\psi(x,t) + u(t) W_1(x)\psi(x,t)} \nonumber \\
& \quad \quad  \quad \quad 
\quad \quad  \quad \quad  \quad   + u^2(t) W_2(x) \psi(x,t),\quad \quad \label{EQ_bilinear_compact_borne}
\end{align}
with $x$ in $\Omega$ and $t$ in $\mathbf{R}$.
Here $\Hc=L^2(\Omega,\mathbf{C})$, and $H(0) = -\Delta + V(x)$.
By Kato-Rellich theorem, the domain $D(H(0))$ of $H(0)$ is equal to the domain of the Laplacian
$H^2_{(0)}=\{\psi \in H^2(\Omega,\mathbf{C})|\psi_{|\partial \Omega}=\Delta \psi_{|\partial 
\Omega}=0\}$ if $\Omega$ is a bounded domain of $\mathbf{R}^n$ 
and it is equal to $H^2(\Omega,\mathbf{C})$ if $\Omega$ is compact manifold.
The operator $H(1) -H(0)= W_1(x) + W_2(x)$ is bounded from $\Hc$ to $\Hc$. 

The existence of a nonresonant chain of connectedness is a generic property~\cite[Theorem 3.4]{genericity-mario-paolo} for systems of the form~\eqref{EQ_bilinear_compact_borne}. 
By Lemma~\ref{prop:nonresonante} system~\eqref{EQ_bilinear_compact_borne}
is exactly controllable in projections for generic control potential $W_1$ and $W_2$.

\section{Conclusions and perspectives}
We presented a sufficient condition for the exact controllability in projections of the linear Schr\"odinger equation with an Hamiltonian that is a bounded perturbation of a free Hamiltonian with pure point spectrum. Most of the results in literature focus on the bilinear case and, indeed, there are several technical challenges arising from the nonlinearity of the Hamiltonian $H(u)$ with respect to the control $u$. 

The condition that $H(1)-H(0)=iB$ is bounded is a quite strong technical assumption. Indeed, this assumption is needed mainly to infer continuity of the propagators, see Remark~\ref{rk:continuitybilinear}, which is crucial in the application of the topological degree argument of Lemma~\ref{LEM_stabilite_structurelle_surjection}. Sufficient conditions for the continuity of the propagator of linear systems  are technically involved and may be hard to check in practice on physical examples (see for instance~\cite{UP}). 

A natural extension of the controllability result in this paper is the analysis of the controllability with several controls. In many examples controllability cannot be achieved with a single scalar control as a consequence of the symmetries of the system.  This happens, for instance, for a planar rotating molecule controlled by one external field only~\cite{quadratique}, which is not controllable.  The Lie--Galerkin condition has been introduced in~\cite{BCS14} exactly to tackle this challenge in the bilinear case.

Finally a challenging perspective is the case in which the free Hamiltonian presents also a continuous spectrum. The Lie--Galerkin methods, indeed, relies on
the lack of coupling between the first $n$ eigenstates and the rest of the spectrum which ensures
 the invariance of the evolution on a suitable Galerkin approximation. 
In order to tackle such a challenge one has to adapt the Lie-Galerkin condition to Hamiltonians having finitely many distinct eigenvalues and estimate the loss of population to the continuous part of the spectrum.

\begin{acknowledgement}
This work is part of the project CONSTAT, supported by the Conseil
R\'egional de Bourgogne Franche Comt\'e and the European Union through the
PO FEDER Bourgogne 2014/2020 programs,  by the French ANR through the
grant QUACO (17-CE40-0007-01) and by EIPHI Graduate School (ANR-17-
EURE-0002). Partially supported by the MIUR Excellence Department Project MatMod@TOV of the Department of Mathematics, University of Rome Tor Vergata CUP E83C23000330006
\end{acknowledgement}

\bibliographystyle{alpha}
\bibliography{biblioteca}

\end{document}